\documentclass[fleqn,reqno]{amsart}
\usepackage{url}
\usepackage{amssymb,amsmath,amssymb,color,cite}
\usepackage{amsmath,amsthm,amssymb, amstext,}
\usepackage{graphicx}
\usepackage{multirow}
\newtheorem{theorem}{Theorem}[section]
\newtheorem{lemma}{Lemma}[section]

\newtheorem{Conjecture}{Conjecture}[section]
\newtheorem{open problem}{\sc \bf Open problem}[section]
\newtheorem{remark}{Remark}[section]

\numberwithin{equation}{section} 
\begin{document}

\title[On the Wilker and Huygens-type inequalities]
   {On the Wilker and Huygens-type inequalities}

\author[{ C.-P. Chen }]{ Chao-Ping Chen$^{*}$}

\address{C.-P. Chen: School of Mathematics and Informatics, Henan Polytechnic University, Jiaozuo City 454000, Henan Province, China}
 \email{chenchaoping@sohu.com}

\author[{ R.B. Paris }]{Richard B. Paris}

\address{R.B. Paris: Division of Computing and Mathematics,\\
 University of Abertay, Dundee, DD1 1HG, UK}
 \email{R.Paris@abertay.ac.uk}

\thanks{*Corresponding Author}



\begin{abstract}
Chen and  Cheung [C.-P. Chen, W.-S. Cheung, Sharpness of Wilker and Huygens type inequalities, J. Inequal. Appl. 2012 (2012) 72,
\url{http://dx.doi.org/10.1186/1029-242X-2012-72}] established sharp Wilker and Huygens-type inequalities.  These authors also proposed three conjectures on Wilker and Huygens-type inequalities.
In this paper, we consider these conjectures. We also present sharp Wilker and Huygens-type inequalities.
\end{abstract}

\maketitle

\section{\bf Introduction}
Wilker  \cite{WilkerE3306} proposed the following two open problems:
\par
(a) Prove that if $0<x<\pi/2$, then
\begin{equation}\label{Wilker-1}
\left(\frac{\sin x}{x}\right)^{2}+\frac{\tan x}{x}>2.
\end{equation}
\par
(b) Find the largest constant c such that
\begin{equation*}\label{Wilker-2}
\left(\frac{\sin x}{x}\right)^{2}+\frac{\tan x}{x}>2+cx^3\tan x
\end{equation*}
for $0<x<\pi/2$.
In \cite{Sumner264--267}, the inequality \eqref{Wilker-1} was proved,
and the following inequality
\begin{equation}\label{Wilker-II}
2+\left(\frac{2}{\pi}\right)^{4}x^3\tan x<\left(\frac{\sin
x}{x}\right)^{2}+\frac{\tan x}{x}<2+\frac{8}{45}x^3\tan x,\qquad
 0<x<\frac{\pi}{2}
\end{equation}
was also established,  where the constants $(2/\pi)^{4}$ and
$\frac{8}{45}$ are the best possible.

The Wilker-type inequalities \eqref{Wilker-1} and \eqref{Wilker-II} have
attracted much interest of many mathematicians and have motivated a
large number of research papers involving different proofs, various
generalizations and improvements (cf.
\cite{Baricz397--406,Chen-Cheung,Chen-CheungJIA,Chen-Sandor55--67,Guo19--22,Mortitc535--541,Mortitc516--520,Neuman399--407,Neuman271--279,Neuman211--217,Neuman715--723,Pinelis905--909,
Wu683--687,Wu447--458,Wu529--535,Wu757--765,Zhang149-151,Zhu749--750,ZhuID4858422,Zhu1998--2004,ZhuID130821}
and the references cited therein).

A related inequality that is of interest to us is Huygens'
inequality \cite{Huygens}, which asserts that
\begin{equation}\label{Huygens}
2\left(\frac{\sin x}{x}\right)+\frac{\tan x}{x}>3,\qquad 0<|x|<\frac{\pi}{2}.
\end{equation}
Wu and Srivastava \cite[Lemma 3]{Wu529--535} established another
inequality
\begin{equation}\label{Wilker-2}
\left(\frac{x}{\sin x}\right)^{2}+\frac{x}{\tan x}>2, \qquad 0<|x|<\frac{\pi}{2}.
\end{equation}
Neuman and S\'andor \cite[Theorem
2.3]{Neuman715--723} proved that for $0<|x|<\pi/2$,
\begin{align}\label{NeumanSandorThm2.3}
\frac{\sin x}{x}<\frac{2+\cos x}{3}<\frac{1}{2}\left(\frac{x}{\sin
x}+\cos x\right).
\end{align}
By multiplying both sides of  inequality \eqref{NeumanSandorThm2.3}
by $x/\sin x$, we obtain that for $0<|x|<\pi/2$,
\begin{align}\label{Cusa-re}
\frac{1}{2}\left[\left(\frac{x}{\sin x}\right)^{2}+\frac{x}{\tan
x}\right]>\frac{2(x/\sin x)+x/\tan x}{3}>1.
\end{align}

 Chen and S\'andor \cite{Chen-Sandor55--67} established the following inequality
 chain:
\begin{align}\label{LWsuggest}
&\quad\frac{\left(\sin x/x\right)^{2}+\tan x/x}{2}>\left(\frac{\sin
x}{x}\right)^{2}\left(\frac{\tan x}{x}\right)>\frac{2\left(\sin
x/x\right)+\tan x/x}{3}\notag\\
&> \left(\frac{\sin x}{x}\right)^{2/3}\left(\frac{\tan
x}{x}\right)^{1/3}>\frac{1}{2}\left[\left(\frac{x}{\sin
x}\right)^{2}+\frac{x}{\tan x}\right]>\frac{2(x/\sin x)+x/\tan
x}{3}>1
\end{align}
for $0<|x|<\pi/2$.

In analogy with \eqref{Wilker-II}, Chen and Cheung \cite{Chen-CheungJIA} established sharp Wilker and Huygens-type inequalities. For example, these authors proved that
for $0<x<\pi/2$,
\begin{equation}\label{Wilker-generalization-ineq-1}
2+\frac{8}{45}x^4+\frac{16}{315}x^5\tan x<\left(\frac{\sin
x}{x}\right)^{2}+\frac{\tan
x}{x}<2+\frac{8}{45}x^4+\left(\frac{2}{\pi}\right)^{6}x^5\tan x,
\end{equation}
where the constants $\frac{16}{315}$ and $\left(2/\pi\right)^{6}$
are best possible,
\begin{align}\label{secondWilker-1}
\left(\frac{x}{\sin x}\right)^{2}+\frac{x}{\tan x}<2+\frac{2}{45}
x^3\tan x,
\end{align}
where the constant $\frac{2}{45}$ is best possible, and
\begin{equation}\label{Huygens-generalization-ineq-1}
3+\frac{3}{20}x^3\tan x<2\left(\frac{\sin x}{x}\right)+\frac{\tan
x}{x}<3+\left(\frac{2}{\pi}\right)^{4}x^3\tan x,
\end{equation}
where the constants $\frac{3}{20}$ and $\left(2/\pi\right)^{4}$
are best possible.

In view of \eqref{Wilker-generalization-ineq-1}, \eqref{secondWilker-1} and \eqref{Huygens-generalization-ineq-1},
Chen and Cheung \cite{Chen-CheungJIA} posed the following conjectures.
\begin{Conjecture}\label{WilkerInequality-Conjecture}
For $0<x<\pi/2$ and $n\geq3$,
\begin{align*}
&2+\sum_{k=3}^{n}\frac{\big(2(2^{2k}-1)|B_{2k}|-(-1)^{k}\big)2^{2k-1}}{(2k)!}x^{2k-2}\\
&\quad +\frac{\big(2(2^{2n+2}-1)|B_{2n+2}|-(-1)^{n+1}\big)2^{2n+1}}{(2n+2)!}x^{2n-1}\tan x\\
&\qquad<\left(\frac{\sin x}{x}\right)^{2}+\frac{\tan x}{x}\\
&\qquad<2+\sum_{k=3}^{n}\frac{\big(2(2^{2k}-1)|B_{2k}|-(-1)^{k}\big)2^{2k-1}}{(2k)!}x^{2k-2}+\left(\frac{2}{\pi}\right)^{2n}x^{2n-1}\tan x,
\end{align*}
where $B_{n}$ $(n\in\mathbb{N}_0, \mathbb{N}_0=\mathbb{N}\cup\{0\},  \mathbb{N}:=\{1, 2, \ldots\})$ are  the Bernoulli numbers, defined by
\begin{equation*}
\begin{split}
\frac{t}{e^{t}-1}=\sum_{n=0}^{\infty}B_{n}\frac{t^{n}}{n!},\qquad |t|<2\pi.
\end{split}
\end{equation*}
\end{Conjecture}
\begin{Conjecture}\label{second-Wilker-Conjecture}
For $0<x<\pi/2$ and $n\geq1$,
\begin{align*}
\left(\frac{x}{\sin x}\right)^{2}+\frac{x}{\tan x}
&<2+\sum_{k=2}^{n}\frac{(k-1)\cdot2^{2k+1}|B_{2k}|}{(2k)!}x^{2k}+\frac{n\cdot2^{2n+3}|B_{2(n+1)|}}{(2n+2)!}x^{2n+1}\tan x.
\end{align*}
Here, and throughout this paper, an empty sum is understood to be zero.
\end{Conjecture}
\begin{Conjecture}\label{Huygens-Inequality-Conjecture}
For $0<x<\pi/2$ and $n\geq2$,
\begin{align*}
&3+\sum_{k=3}^{n}\left(\frac{2^{2k}(2^{2k}-1)|B_{2k}|}{4k}-(-1)^k\right)\frac{2}{(2k-1)!}x^{2k-2}\notag\\
&\qquad+\left(\frac{2^{2n+2}(2^{2n+2}-1)|B_{2n+2}|}{4(n+1)}-(-1)^{n+1}\right)\frac{2}{(2n+1)!}x^{2n-1}\tan x\notag\\
&\quad<2\left(\frac{\sin x}{x}\right)+\frac{\tan x}{x}\\
&\quad<3+\sum_{k=3}^{n}\left(\frac{2^{2k}(2^{2k}-1)|B_{2k}|}{4k}-(-1)^k\right)\frac{2}{(2k-1)!}x^{2k-2}+\left(\frac{2}{\pi}\right)^{2n}x^{2n-1}\tan
x.
\end{align*}
\end{Conjecture}
 Recently, Chen and Paris \cite{Chen-Paris-submissionII} proved Conjecture \ref{second-Wilker-Conjecture}. This paper is a continuation of our earlier work \cite{Chen-Paris-submissionII}.
The first aim of the present paper is to prove Conjectures \ref{WilkerInequality-Conjecture} and \ref{Huygens-Inequality-Conjecture}.

Mortici \cite[Theorem 1]{Mortitc516--520} presented  the following double inequality:
\begin{align}\label{Wilker-Inequality-Mortici}
2+\left(\frac{8}{45}-\frac{8}{945}x^2\right)x^{3}\tan x&<\left(\frac{\sin x}{x}\right)^{2}+\frac{\tan x}{x}\nonumber\\
&<2+\left(\frac{8}{45}-\frac{8}{945}x^2+\frac{16}{14175}x^{4}\right)x^{3}\tan x, \qquad 0<x<1.
\end{align}
By using Maple software, we find that
\begin{align}\label{Wilker-Inequality-Mortici-series}
\frac{\left(\frac{\sin x}{x}\right)^{2}+\frac{\tan x}{x}-2}{x^{3}\tan x}&=\frac{8}{45}-\frac{8}{945}x^2+\frac{16}{14175}x^{4}+\frac{8}{467775} x^6+\frac{3184}{638512875} x^8\nonumber\\
&\quad+\frac{272}{638512875} x^{10}+\frac{7264}{162820783125} x^{12}+\cdots.
\end{align}
This fact   led us to claim that the upper bound in \eqref{Wilker-Inequality-Mortici} should be the lower bound.
  The second aim of the present paper is to  prove the following inequality:
\begin{align}\label{Wilker-Mortici-Inequality}
&2+\left(\frac{8}{45}-\frac{8}{945}x^2+\frac{16}{14175}x^{4}\right)x^{3}\tan x<\left(\frac{\sin x}{x}\right)^{2}+\frac{\tan x}{x}\nonumber\\
&\quad<2+\left(\frac{8}{45}-\frac{8}{945}x^2+\frac{241920-2688\pi^4+32\pi^6}{945\pi^8}x^{4}\right)x^{3}\tan x,\qquad 0<x<\frac{\pi}{2},
\end{align}
where the constants $\frac{16}{14175}$ and $(241920-2688\pi^4+32\pi^6)/(945\pi^8)$ are the best possible.
\begin{remark}
The inequalities \eqref{Wilker-Mortici-Inequality} are sharper than the inequalities \eqref{Wilker-II} and \eqref{Wilker-generalization-ineq-1}.
\end{remark}
In analogy with  \eqref{Wilker-Mortici-Inequality},  we here determine the best possible constants
$\alpha, \beta, \lambda, \mu, \rho$, and $\varrho$ such that
\begin{align*}
2+\left(\frac{2}{45}-\frac{2}{315}x^2-\alpha x^{4}\right)x^{3}\tan x&<\left(\frac{x}{\sin x}\right)^{2}+\frac{x}{\tan x}<2+\left(\frac{2}{45}-\frac{2}{315}x^2-\beta x^{4}\right)x^{3}\tan x,
\end{align*}
\begin{align*}
3+\left(\frac{3}{20}+\frac{1}{280}x^2+\lambda x^{4}\right)x^{3}\tan x&<2\left(\frac{\sin x}{x}\right)+\frac{\tan x}{x}<3+\left(\frac{3}{20}+\frac{1}{280}x^2+\mu x^{4}\right)x^{3}\tan x
\end{align*}
and
\begin{align*}
3+\left(\frac{1}{60}-\frac{1}{280}x^2-\rho  x^{4}\right)x^{3}\tan x&<2\left(\frac{x}{\sin x}\right)+\frac{x}{\tan x}<3+\left(\frac{1}{60}-\frac{1}{280}x^2-\varrho  x^{4}\right)x^{3}\tan x
\end{align*}
for $0<x<\pi/2$. This is the last aim of the present paper.

\vskip 8mm

\section{\bf A useful lemma}
It is well known that
 \begin{align}\label{tan-series}
\tan x=\sum_{n=1}^\infty\frac{2^{2n}(2^{2n}-1)|B_{2n}|}{(2n)!}x^{2n-1},\qquad |x|<\frac{\pi}2,
\end{align}
By using induction, Chen and Qi \cite{Chen-Qi} (see also \cite{Zhao499--506}) proved the following
 \begin{lemma}
 Let $n\geq1$ be an integer. Then for $0<x<\pi/2$, we have
\begin{align}\label{mainr}
\frac{2^{2n+2}(2^{2n+2}-1){|B_{2n+2}|}}{(2n+2)!}x^{2n}\tan x
&<\tan x-\sum_{k=1}^{n}\frac{2^{2k}(2^{2k}-1)|B_{2k}|}{(2k)!}x^{2k-1}\nonumber\\
&<\left(\frac2\pi\right)^{2n}x^{2n}\tan x,
\end{align}
where the the constants
\begin{equation*}
\frac{2^{2n+2}(2^{2n+2}-1){|B_{2n+2}|}}{(2n+2)!}\quad\text{and}\quad \left(\frac2\pi\right)^{2n}
\end{equation*}
are the best possible.
\end{lemma}

\vskip 4mm

\section{\bf Main results}

\begin{theorem}\label{ThmWilker-Conjecture}
For $0<x<\pi/2$ and $n\geq3$, we have
\begin{align}\label{Wilker-Inequality-Conjecture}
&\left(\frac{(-1)^{n}2^{2n+1}}{(2n+2)!}+\frac{2^{2n+2}(2^{2n+2}-1)|B_{2n+2}|}{(2n+2)!}\right)x^{2n-1}\tan x\nonumber\\
&\quad<\left(\frac{\sin x}{x}\right)^{2}+\frac{\tan x}{x}-\left\{2+\sum_{k=3}^{n}\left(\frac{(-1)^{k-1}2^{2k-1}}{(2k)!}+\frac{2^{2k}(2^{2k}-1)|B_{2k}|}{(2k)!}\right)x^{2k-2}\right\}\nonumber\\
&\quad<\left(\frac{2}{\pi}\right)^{2n}x^{2n-1}\tan x.
\end{align}
\end{theorem}

\begin{proof}
First of all, we prove the first inequality in \eqref{Wilker-Inequality-Conjecture}.
By using the power series expansions for $\cos x$  and $\tan x$, we have
\begin{align*}
\left(\frac{\sin x}{x}\right)^{2}+\frac{\tan
x}{x}&=\frac{1}{2x^2}-\frac{1}{2x^2}\cos(2x)+\frac{\tan x}{x}\\
&=2+\sum_{k=3}^{n}\left(\frac{(-1)^{k-1}2^{2k-1}}{(2k)!}+\frac{2^{2k}(2^{2k}-1)|B_{2k}|}{(2k)!}\right)x^{2k-2}+r_n(x),
\end{align*}
where
\begin{align*}
r_n(x)=\sum_{k=n+1}^{\infty}\left(\frac{(-1)^{k-1}2^{2k-1}}{(2k)!}+\frac{2^{2k}(2^{2k}-1)|B_{2k}|}{(2k)!}\right)x^{2k-2}.
\end{align*}

The first inequality in \eqref{Wilker-Inequality-Conjecture} is equivalent to
\begin{align*}
\left(\frac{(-1)^{n}2^{2n+1}}{(2n+2)!}+\frac{2^{2n+2}(2^{2n+2}-1)|B_{2n+2}|}{(2n+2)!}\right)x^{2n-1}\tan x<r_n(x)
\end{align*}
for $0<x<\pi/2$ and $n\geq3$, which can be written by \eqref{tan-series} as
\begin{align*}
&\sum_{k=n+2}^{\infty}\Bigg\{\left(\frac{(-1)^{n}2^{2n+1}}{(2n+2)!}+\frac{2^{2n+2}(2^{2n+2}-1)|B_{2n+2}|}{(2n+2)!}\right)\frac{2^{2k-2n}(2^{2k-2n}-1)|B_{2k-2n}|}{(2k-2n)!}\nonumber\\
&\qquad\qquad-\left(\frac{(-1)^{k-1}2^{2k-1}}{(2k)!}+\frac{2^{2k}(2^{2k}-1)|B_{2k}|}{(2k)!}\right)\Bigg\}x^{2k-2}<0,
\end{align*}
where we note that the term corresponding to $k=n+1$ vanishes.

We claim that  for $k\geq n+2$,
\begin{align}\label{Wilker-Inequality-Conjecture-claim}
&\left(\frac{(-1)^{n}2^{2n+1}}{(2n+2)!}+\frac{2^{2n+2}(2^{2n+2}-1)|B_{2n+2}|}{(2n+2)!}\right)\frac{2^{2k-2n}(2^{2k-2n}-1)|B_{2k-2n}|}{(2k-2n)!}\nonumber\\
&\qquad\qquad<\frac{(-1)^{k-1}2^{2k-1}}{(2k)!}+\frac{2^{2k}(2^{2k}-1)|B_{2k}|}{(2k)!},
\end{align}
It is enough to prove the following inequality:
\begin{align*}
&\left(\frac{2^{2n+1}}{(2n+2)!}+\frac{2^{2n+2}(2^{2n+2}-1)|B_{2n+2}|}{(2n+2)!}\right)\frac{2^{2k-2n}(2^{2k-2n}-1)|B_{2k-2n}|}{(2k-2n)!}\nonumber\\
&\qquad\qquad<-\frac{2^{2k-1}}{(2k)!}+\frac{2^{2k}(2^{2k}-1)|B_{2k}|}{(2k)!},\qquad  k\geq n+2.
\end{align*}

Using the following inequality (see \cite{Daniello329--332}):
 \begin{eqnarray}\label{eq:1.11}
\frac{2}{\left(2\pi\right)^{2n}\left(1-2^{-2n}\right)}<\frac{\left|B_{2n}\right|}{(2n)!}< \frac{2}{\left(2\pi\right)^{2n}\left(1-2^{1-2n}\right)},\qquad
n\geq1,
\end{eqnarray}
it suffices to show that for $k\geq n+2$,
\begin{align*}
&\left(\frac{2^{2n+1}}{(2n+2)!}+\frac{2^{2n+2}(2^{2n+2}-1)2}{\left(2\pi\right)^{2n+2}\left(1-2^{1-2(n+1)}\right)}\right)\frac{2^{2k-2n}(2^{2k-2n}-1)2}{\left(2\pi\right)^{2k-2n}\left(1-2^{1-2(k-n)}\right)}\nonumber\\
&\qquad\qquad<-\frac{2^{2k-1}}{(2k)!}+\frac{2^{2k}(2^{2k}-1)2}{\left(2\pi\right)^{2k}\left(1-2^{-2k}\right)},
\end{align*}
which can be rearranged as
\begin{align*}
&\left(\frac{2^{2n+1}}{(2n+2)!}+\frac{2(2^{2n+2}-1)}{2^{2n+2}-2}\left(\frac{2}{\pi}\right)^{2n+2}\right)\frac{2^{2k-2n}-1}{2^{2k-2n}-2}\left(\frac{2}{\pi}\right)^{2k-2n}+\frac{2^{2k-2}}{(2k)!}<\left(\frac{2}{\pi}\right)^{2k},
\end{align*}
\begin{align*}
&\left(\frac{2^{2n+1}}{(2n+2)!}+2\left(1+\frac{1}{2^{2n+2}-2}\right)\left(\frac{2}{\pi}\right)^{2n+2}\right)\left(1+\frac{1}{2^{2k-2n}-2}\right)\left(\frac{\pi}{2}\right)^{2n}\\
&\qquad\qquad\qquad+\frac{2^{2k-2}}{(2k)!}\left(\frac{\pi}{2}\right)^{2k}<1,\qquad  k\geq n+2.
\end{align*}
Noting that the sequences
\begin{align*}
 1+\frac{1}{2^{2k-2n}-2}  \quad \text{and}\quad \frac{2^{2k-2}}{(2k)!}\left(\frac{\pi}{2}\right)^{2k}
\end{align*}
are both strictly decreasing for $k\geq n+2$, it suffices to show that
\begin{align*}
&\left(\frac{2^{2n+1}}{(2n+2)!}+2\left(1+\frac{1}{2^{2n+2}-2}\right)\left(\frac{2}{\pi}\right)^{2n+2}\right)\frac{15}{14}\left(\frac{\pi}{2}\right)^{2n}+\frac{2^{2n+2}}{(2n+4)!}\left(\frac{\pi}{2}\right)^{2n+4}<1
\end{align*}
for $n\geq3$, which can be rearranged as
\begin{align*}
\frac{2^{2n+1}}{(2n+2)!}\left(\frac{\pi}{2}\right)^{2n+2}+\frac{1}{2^{2n+1}-1}+\frac{2^{2n+2}}{(2n+4)!}\left(\frac{\pi}{2}\right)^{2n+6}\left(\frac{14}{15}\right)<\frac{14}{15}\left(\frac{\pi}{2}\right)^{2}-2,\qquad n\geq3.
\end{align*}
Noting that the sequence
\begin{align*}
a_n:=\frac{2^{2n+1}}{(2n+2)!}\left(\frac{\pi}{2}\right)^{2n+2}+\frac{1}{2^{2n+1}-1}+\frac{2^{2n+2}}{(2n+4)!}\left(\frac{\pi}{2}\right)^{2n+6}\left(\frac{14}{15}\right),\qquad n\geq3
\end{align*}
is strictly decreasing, we see that
\begin{align*}
a_n\leq a_3=\frac{1}{127}+\frac{\pi^8}{80640}+\frac{\pi^{12}}{62208000}<\frac{14}{15}\left(\frac{\pi}{2}\right)^{2}-2
\end{align*}
holds true for $n\geq3$, since
\begin{align*}
\frac{1}{127}+\frac{\pi^8}{80640}+\frac{\pi^{12}}{62208000}=0.14039705\ldots,\quad \frac{14}{15}\left(\frac{\pi}{2}\right)^{2}-2=0.30290769\ldots\ .
\end{align*}
This proves the claim \eqref{Wilker-Inequality-Conjecture-claim}. Hence, the first inequality in \eqref{Wilker-Inequality-Conjecture} holds for $0<x<\pi/2$ and $n\geq3$.

Secondly, we prove the second inequality in \eqref{Wilker-Inequality-Conjecture}.
We consider two cases.

\textit{Case 1.}  $n=2N+1$ ($N\geq1$).

It is well known that for $x\not=0$,
\begin{align*}
\sum_{k=1}^{2N}(-1)^{k-1}\frac{x^{2k-2}}{(2k-2)!}<\cos x<\sum_{k=1}^{2N+1}(-1)^{k-1}\frac{x^{2k-2}}{(2k-2)!}.
\end{align*}
We then obtain that
\begin{align}\label{cos2x-series-inequality}
\left(\frac{\sin x}{x}\right)^{2}&=\frac{1}{2x^2}-\frac{1}{2x^2}\cos(2x)<\sum_{k=1}^{2N+1}(-1)^{k-1}\frac{2^{2k-1}}{(2k)!}x^{2k-2}.
\end{align}
The choice $n=2N+1$ in \eqref{mainr}, we obtain from the right-hand  inequality of \eqref{mainr} that
\begin{equation}\label{mainrRight-n-odd}
\frac{\tan x}{x}<\sum_{k=1}^{2N+1}\frac{2^{2k}(2^{2k}-1)|B_{2k}|}{(2k)!}x^{2k-2}+\left(\frac2\pi\right)^{4N+2}x^{4N+1}\tan x.
\end{equation}
Adding these two expressions, we obtain
\begin{align*}
\left(\frac{\sin x}{x}\right)^{2}+\frac{\tan x}{x}&<2+\sum_{k=3}^{2N+1}\left(\frac{(-1)^{k-1}2^{2k-1}}{(2k)!}+\frac{2^{2k}(2^{2k}-1)|B_{2k}|}{(2k)!}\right)x^{2k-2}\\
&\quad+\left(\frac2\pi\right)^{4N+2}x^{4N+1}\tan x.
\end{align*}
This shows that the second inequality in \eqref{Wilker-Inequality-Conjecture} holds for $n=2N+1$.

\textit{Case 2.}  $n=2N$ ($N\geq2$).

Write
\begin{align*}
\left(\frac{\sin x}{x}\right)^{2}+\frac{\tan x}{x}&=2+\sum_{k=3}^{2N}\left(\frac{(-1)^{k-1}2^{2k-1}}{(2k)!}+\frac{2^{2k}(2^{2k}-1)|B_{2k}|}{(2k)!}\right)x^{2k-2}\\
&\quad+\sum_{k=2N+1}^{\infty}\left(\frac{(-1)^{k-1}2^{2k-1}}{(2k)!}+\frac{2^{2k}(2^{2k}-1)|B_{2k}|}{(2k)!}\right)x^{2k-2}.
\end{align*}
We need to prove
\begin{align}\label{Wilker-conjecture-rightn=2N}
\sum_{k=2N+1}^{\infty}\left(\frac{(-1)^{k-1}2^{2k-1}}{(2k)!}+\frac{2^{2k}(2^{2k}-1)|B_{2k}|}{(2k)!}\right)x^{2k-2}<\left(\frac{2}{\pi}\right)^{4N}x^{4N-1}\tan x.
\end{align}
 Noting that \eqref{tan-series} holds,  we can rewrite \eqref{Wilker-conjecture-rightn=2N} as
\begin{align*}
&\sum_{k=2N+1}^{\infty}\Bigg\{\frac{(-1)^{k-1}2^{2k-1}}{(2k)!}+\frac{2^{2k}(2^{2k}-1)|B_{2k}|}{(2k)!}\\
&\qquad\qquad-\left(\frac{2}{\pi}\right)^{4N}\frac{2^{2k-4N}(2^{2k-4N}-1)|B_{2k-4N}|}{(2k-4N)!}\Bigg\}x^{2k-2}<0.
\end{align*}

We claim that  for $k\geq 2N+1$,
\begin{align}\label{Wilker-conjecture-rightnclaim}
\frac{(-1)^{k-1}2^{2k-1}}{(2k)!}+\frac{2^{2k}(2^{2k}-1)|B_{2k}|}{(2k)!}<\left(\frac{2}{\pi}\right)^{4N}\frac{2^{2k-4N}(2^{2k-4N}-1)|B_{2k-4N}|}{(2k-4N)!}.
\end{align}
It is enough to prove the following inequality:
\begin{align}\label{Wilker-conjecture-rightn-claimproof}
\frac{2^{2k-1}}{(2k)!}+\frac{2^{2k}(2^{2k}-1)|B_{2k}|}{(2k)!}<\left(\frac{2}{\pi}\right)^{4N}\frac{2^{2k-4N}(2^{2k-4N}-1)|B_{2k-4N}|}{(2k-4N)!},\quad k\geq 2N+1.
\end{align}

Using \eqref{eq:1.11}, we find that for $k\geq2N+1$,
\begin{align}\label{Monotinic-ratio1}
\frac{\dfrac{2^{2k+2}(2^{2k+2}-1)|B_{2k+2}|}{(2k+2)!}}{\dfrac{2^{2k}(2^{2k}-1)|B_{2k}|}{(2k)!}}<\frac{\dfrac{2^{2k+2}(2^{2k+2}-1)2}{\left(2\pi\right)^{2k+2}\left(1-2^{1-2(k+1)}\right)}}{\dfrac{2^{2k}(2^{2k}-1)2}{\left(2\pi\right)^{2k}\left(1-2^{-2k}\right)}}=\frac{2(4^k-2)(4\cdot4^k-1)}{\pi^2(4^k-1)(2\cdot4^k-1)}<1
\end{align}
and\footnote{The inequality \eqref{Monotinic-ratio2} is proved in the appendix.}
\begin{align}\label{Monotinic-ratio2}
&\dfrac{\dfrac{2^{2k-4N+2}(2^{2k-4N+2}-1)|B_{2k-4N+2}|}{(2k-4N+2)!}}{\dfrac{2^{2k-4N}(2^{2k-4N}-1)|B_{2k-4N}|}{(2k-4N)!}}>\frac{\dfrac{2^{2k-4N+2}(2^{2k-4N+2}-1)2}{\left(2\pi\right)^{2k-4N+2}\left(1-2^{-2(k-2N+1)}\right)}}{ \dfrac{2^{2k-4N}(2^{2k-4N}-1)2}{\left(2\pi\right)^{2k-4N}\left(1-2^{1-2(k-2N)}\right)}}\nonumber\\
&\quad=\frac{16^{k+N+1}-(8\cdot256^N+64^N)4^{k+1}+8\cdot1024^N}{\pi^2(4^k-16^N)(4^{k+1}-16^N)}>1.
\end{align}
Hence, the sequence
\begin{align*}
\frac{2^{2k-1}}{(2k)!}+\frac{2^{2k}(2^{2k}-1)|B_{2k}|}{(2k)!}
\end{align*}
is strictly decreasing, and the sequence
\begin{align*}
\left(\frac{2}{\pi}\right)^{2N}\frac{2^{2k-2N}(2^{2k-2N}-1)|B_{2k-2N}|}{(2k-2N)!}
\end{align*}
is strictly increasing for $k\geq2N+1$. In order to prove  \eqref{Wilker-conjecture-rightn-claimproof}, it suffices to show that for $k\geq2N+1$,
\begin{align}\label{Wilker-conjecture-rightn-claimproofshow}
\frac{2^{4N+1}}{(4N+2)!}+\frac{2^{4N+2}(2^{4N+2}-1)|B_{4N+2}|}{(4N+2)!}<\left(\frac{2}{\pi}\right)^{4N}\frac{2^{2}(2^{2}-1)|B_{2}|}{2!}=\left(\frac{2}{\pi}\right)^{4N}.
\end{align}
By \eqref{eq:1.11}, it suffices to show that
\begin{align*}
\frac{2^{4N+1}}{(4N+2)!}+\frac{2^{4N+2}(2^{4N+2}-1)2}{\left(2\pi\right)^{4N+2}\left(1-2^{1-2(2N+1)}\right)}<\left(\frac{2}{\pi}\right)^{4N},
\end{align*}
\begin{align*}
\frac{2^{4N+1}}{(4N+2)!}+2\left(1+\frac{1}{2^{4N+2}-2}\right)\left(\frac{2}{\pi}\right)^{4N+2}<\left(\frac{2}{\pi}\right)^{4N},
\end{align*}
\begin{align*}
\frac{2^{4N+1}}{(4N+2)!}\left(\frac{\pi}{2}\right)^{4N+2}+\frac{1}{2^{4N+1}-1}<\left(\frac{\pi}{2}\right)^{2}-2,\qquad  N\geq2.
\end{align*}
Noting that the sequence
\begin{align*}
b_N:=\frac{2^{4N+1}}{(4N+2)!}\left(\frac{\pi}{2}\right)^{4N+2}+\frac{1}{2^{4N+1}-1},\qquad  N\geq2
\end{align*}
is strictly decreasing, we see that
\begin{align*}
b_N\leq b_2=\frac{1}{511}+\frac{\pi^{10}}{7257600}<\left(\frac{\pi}{2}\right)^{2}-2
\end{align*}
holds true for $N\geq2$, since
\begin{align*}
\frac{1}{511}+\frac{\pi^{10}}{7257600}=0.0148603\ldots, \quad \left(\frac{\pi}{2}\right)^{2}-2=0.4674011\ldots\ .
\end{align*}
 This proves the claim \eqref{Wilker-conjecture-rightnclaim}. Hence, \eqref{Wilker-conjecture-rightn=2N} holds, which shows that the second inequality in \eqref{Wilker-Inequality-Conjecture} holds for $n=2N$.
Thus, the second inequality in \eqref{Wilker-Inequality-Conjecture} holds for $0<x<\pi/2$ and $n\geq3$. The proof of Theorem \ref{ThmWilker-Conjecture} is complete.
\end{proof}

\begin{theorem}\label{Thm-Huygens-Inequality-Conjecture}
For $0<x<\pi/2$ and $n\geq2$, we have
\begin{align}\label{ThmHuygensInequalityConjecture}
&\left(\frac{2(-1)^{n}}{(2n+1)!}+\frac{2^{2n+2}(2^{2n+2}-1)|B_{2n+2}|}{(2n+2)!}\right)x^{2n-1}\tan x\notag\\
&\quad<2\left(\frac{\sin x}{x}\right)+\frac{\tan x}{x}-\left\{3+\sum_{k=3}^{n}\left(\frac{2(-1)^{k-1}}{(2k-1)!}+\frac{2^{2k}(2^{2k}-1)|B_{2k}|}{(2k)!}\right)x^{2k-2}\right\}\notag\\
&\quad<\left(\frac{2}{\pi}\right)^{2n}x^{2n-1}\tan x.
\end{align}
\end{theorem}

\begin{proof}
First of all, we prove the first inequality in \eqref{ThmHuygensInequalityConjecture}.
By using the power series expansions for $\sin x$  and $\tan x$, we have
\begin{align*}
2\left(\frac{\sin x}{x}\right)+\frac{\tan x}{x}=3+\sum_{k=3}^{n}\left(\frac{2(-1)^{k-1}}{(2k-1)!}+\frac{2^{2k}(2^{2k}-1)|B_{2k}|}{(2k)!}\right)x^{2k-2}+R_n(x),
\end{align*}
where
\begin{align*}
R_n(x)=\sum_{k=n+1}^{\infty}\left(\frac{2(-1)^{k-1}}{(2k-1)!}+\frac{2^{2k}(2^{2k}-1)|B_{2k}|}{(2k)!}\right)x^{2k-2}.
\end{align*}

The first inequality in \eqref{ThmHuygensInequalityConjecture} is equivalent to
\begin{align*}
&\left(\frac{2(-1)^{n}}{(2n+1)!}+\frac{2^{2n+2}(2^{2n+2}-1)|B_{2n+2}|}{(2n+2)!}\right)x^{2n-1}\tan x<R_n(x)
\end{align*}
for $0<x<\pi/2$ and $n\geq2$, which can be written by  \eqref{tan-series} as
\begin{align*}
&\sum_{k=n+2}^{\infty}\Bigg\{\left(\frac{2(-1)^{n}}{(2n+1)!}+\frac{2^{2n+2}(2^{2n+2}-1)|B_{2n+2}|}{(2n+2)!}\right)\frac{2^{2k-2n}(2^{2k-2n}-1)|B_{2k-2n}|}{(2k-2n)!}\nonumber\\
&\qquad\qquad-\left(\frac{2(-1)^{k-1}}{(2k-1)!}+\frac{2^{2k}(2^{2k}-1)|B_{2k}|}{(2k)!}\right)\Bigg\}x^{2k-2}<0,
\end{align*}
where we note that the term corresponding to $k=n+1$ vanishes.

We claim that  for $k\geq n+2$,
\begin{align}\label{ThmHuygensInequalityConjectureProofre}
&\left(\frac{2(-1)^{n}}{(2n+1)!}+\frac{2^{2n+2}(2^{2n+2}-1)|B_{2n+2}|}{(2n+2)!}\right)\frac{2^{2k-2n}(2^{2k-2n}-1)|B_{2k-2n}|}{(2k-2n)!}\nonumber\\
&\qquad\qquad<\frac{2(-1)^{k-1}}{(2k-1)!}+\frac{2^{2k}(2^{2k}-1)|B_{2k}|}{(2k)!}.
\end{align}
It is enough to prove the following inequality:
\begin{align*}
&\left(\frac{2}{(2n+1)!}+\frac{2^{2n+2}(2^{2n+2}-1)|B_{2n+2}|}{(2n+2)!}\right)\frac{2^{2k-2n}(2^{2k-2n}-1)|B_{2k-2n}|}{(2k-2n)!}\nonumber\\
&\qquad\qquad<-\frac{2}{(2k-1)!}+\frac{2^{2k}(2^{2k}-1)|B_{2k}|}{(2k)!},\qquad  k\geq n+2.
\end{align*}
Using \eqref{eq:1.11}, it suffices to show that for $k\geq n+2$,
\begin{align*}
&\left(\frac{2}{(2n+1)!}+\frac{2^{2n+2}(2^{2n+2}-1)2}{\left(2\pi\right)^{2n+2}\left(1-2^{1-2(n+1)}\right)}\right)\frac{2^{2k-2n}(2^{2k-2n}-1)2}{\left(2\pi\right)^{2k-2n}\left(1-2^{1-2(k-n)}\right)}\nonumber\\
&\qquad\qquad<-\frac{2}{(2k-1)!}+\frac{2^{2k}(2^{2k}-1)2}{\left(2\pi\right)^{2k}\left(1-2^{-2k}\right)},
\end{align*}
which can be rearranged as
\begin{align*}
&\left(\frac{1}{(2n+1)!}+\frac{2^{2n+2}-1}{2^{2n+2}-2}\left(\frac{2}{\pi}\right)^{2n+2}\right)\frac{2(2^{2k-2n}-1)}{2^{2k-2n}-2}\left(\frac{2}{\pi}\right)^{2k-2n}+\frac{1}{(2k-1)!}<\left(\frac{2}{\pi}\right)^{2k},
\end{align*}
\begin{align*}
&\left(\frac{1}{(2n+1)!}+\left(1+\frac{1}{2^{2n+2}-2}\right)\left(\frac{2}{\pi}\right)^{2n+2}\right)2\left(1+\frac{1}{2^{2k-2n}-2}\right)\left(\frac{\pi}{2}\right)^{2n}\\
&\qquad\qquad\qquad+\frac{1}{(2k-1)!}\left(\frac{\pi}{2}\right)^{2k}<1,\qquad  k\geq n+2.
\end{align*}
Noting that the sequences
\begin{align*}
   2\left(1+\frac{1}{2^{2k-2n}-2}\right)      \quad \text{and}\quad \frac{1}{(2k-1)!}\left(\frac{\pi}{2}\right)^{2k}
\end{align*}
are both strictly decreasing for $k\geq n+2$, it suffices to show that
\begin{align*}
\left(\frac{1}{(2n+1)!}+\left(1+\frac{1}{2^{2n+2}-2}\right)\left(\frac{2}{\pi}\right)^{2n+2}\right)\frac{15}{7}\left(\frac{\pi}{2}\right)^{2n}+\frac{1}{(2n+3)!}\left(\frac{\pi}{2}\right)^{2n+4}<1
\end{align*}
for $n\geq2$, which can be rearranged as
\begin{align*}
\frac{1}{(2n+1)!}\left(\frac{\pi}{2}\right)^{2n+2}+\frac{1}{2^{2n+2}-2}+\frac{1}{(2n+3)!}\left(\frac{\pi}{2}\right)^{2n+6}\left(\frac{7}{15}\right)<\frac{7}{15}\left(\frac{\pi}{2}\right)^{2}-1,\qquad n\geq2.
\end{align*}
Noting that the sequence
\begin{align*}
x_n:=\frac{1}{(2n+1)!}\left(\frac{\pi}{2}\right)^{2n+2}+\frac{1}{2^{2n+2}-2}+\frac{1}{(2n+3)!}\left(\frac{\pi}{2}\right)^{2n+6}\left(\frac{7}{15}\right),\qquad n\geq2
\end{align*}
is strictly decreasing, we see that
\begin{align*}
&x_n\leq x_2=\frac{1}{62}+\frac{\pi^6}{7680}+\frac{\pi^{10}}{11059200}<\frac{7}{15}\left(\frac{\pi}{2}\right)^{2}-1
\end{align*}
holds true for $n\geq2$, since
\begin{align*}
\frac{1}{62}+\frac{\pi^6}{7680}+\frac{\pi^{10}}{11059200}=0.1497778\ldots,\quad  \frac{7}{15}\left(\frac{\pi}{2}\right)^{2}-1=0.1514538\ldots\ .
\end{align*}
 This proves the claim \eqref{ThmHuygensInequalityConjectureProofre}. Hence, the first inequality in \eqref{ThmHuygensInequalityConjecture} holds for $0<x<\pi/2$ and $n\geq2$.

Secondly, we prove the second inequality in \eqref{ThmHuygensInequalityConjecture}.
We consider two cases.

\textit{Case 1.}  $n=2N+1$ ($N\geq1$).

It is well known that for $x\not=0$,
\begin{align}\label{sin-series-inequality}
\sum_{k=1}^{2N}(-1)^{k-1}\frac{x^{2k-2}}{(2k-1)!}<\frac{\sin x}{x}<\sum_{k=1}^{2N+1}(-1)^{k-1}\frac{x^{2k-2}}{(2k-1)!}.
\end{align}
From the second inequality in  \eqref{sin-series-inequality} and
\eqref{mainrRight-n-odd}, we obtain
\begin{align*}
2\left(\frac{\sin x}{x}\right)+\frac{\tan x}{x}
&<3+\sum_{k=3}^{2N+1}\left(\frac{2(-1)^{k-1}}{(2k-1)!}+\frac{2^{2k}(2^{2k}-1)|B_{2k}|}{(2k)!}\right)x^{2k-2}\\
&\quad+\left(\frac2\pi\right)^{4N+2}x^{4N+1}\tan x.
\end{align*}
This shows that the second inequality in \eqref{ThmHuygensInequalityConjecture} holds for $n=2N+1$.

\textit{Case 2.}  $n=2N$ ($N\geq1$).

Write
\begin{align*}
2\left(\frac{\sin x}{x}\right)+\frac{\tan x}{x}
&=3+\sum_{k=3}^{2N}\left(\frac{2(-1)^{k-1}}{(2k-1)!}+\frac{2^{2k}(2^{2k}-1)|B_{2k}|}{(2k)!}\right)x^{2k-2}\\
&\quad+\sum_{k=2N+1}^{\infty}\left(\frac{2(-1)^{k-1}}{(2k-1)!}+\frac{2^{2k}(2^{2k}-1)|B_{2k}|}{(2k)!}\right)x^{2k-2}.
\end{align*}
We need to prove
\begin{align}\label{Huygens-conjecture-rightn=2N}
\sum_{k=2N+1}^{\infty}\left(\frac{2(-1)^{k-1}}{(2k-1)!}+\frac{2^{2k}(2^{2k}-1)|B_{2k}|}{(2k)!}\right)x^{2k-2}<\left(\frac{2}{\pi}\right)^{4N}x^{4N-1}\tan x.
\end{align}
 Noting that \eqref{tan-series} holds,  we can rewrite \eqref{Huygens-conjecture-rightn=2N} as
\begin{align*}
&\sum_{k=2N+1}^{\infty}\Bigg\{\frac{2(-1)^{k-1}}{(2k-1)!}+\frac{2^{2k}(2^{2k}-1)|B_{2k}|}{(2k)!}\\
&\qquad\qquad-\left(\frac{2}{\pi}\right)^{4N}\frac{2^{2k-4N}(2^{2k-4N}-1)|B_{2k-4N}|}{(2k-4N)!}\Bigg\}x^{2k-2}<0.
\end{align*}

We claim that  for $k\geq 2N+1$,
\begin{align}\label{Huygens-conjecture-rightnclaim}
\frac{2(-1)^{k-1}}{(2k-1)!}+\frac{2^{2k}(2^{2k}-1)|B_{2k}|}{(2k)!}<\left(\frac{2}{\pi}\right)^{4N}\frac{2^{2k-4N}(2^{2k-4N}-1)|B_{2k-4N}|}{(2k-4N)!}.
\end{align}
It is enough to prove the following inequality:
\begin{align}\label{Huygens-conjecture-rightn-claimproof}
\frac{2}{(2k-1)!}+\frac{2^{2k}(2^{2k}-1)|B_{2k}|}{(2k)!}<\left(\frac{2}{\pi}\right)^{4N}\frac{2^{2k-4N}(2^{2k-4N}-1)|B_{2k-4N}|}{(2k-4N)!},\quad k\geq 2N+1.
\end{align}

By \eqref{Monotinic-ratio1} and \eqref{Monotinic-ratio2}, we see that the sequence
\begin{align*}
\frac{2}{(2k-1)!}+\frac{2^{2k}(2^{2k}-1)|B_{2k}|}{(2k)!}
\end{align*}
is strictly decreasing, and the sequence
\begin{align*}
\left(\frac{2}{\pi}\right)^{2N}\frac{2^{2k-2N}(2^{2k-2N}-1)|B_{2k-2N}|}{(2k-2N)!}
\end{align*}
is strictly increasing for $k\geq2N+1$. In order to prove  \eqref{Huygens-conjecture-rightn-claimproof}, it suffices to show that for $k\geq2N+1$,
\begin{align}\label{Huygens-conjecture-rightn-claimProofsuffices}
\frac{2}{(4N+1)!}+\frac{2^{4N+2}(2^{4N+2}-1)|B_{4N+2}|}{(4N+2)!}<\left(\frac{2}{\pi}\right)^{4N}\frac{2^{2}(2^{2}-1)|B_{2}|}{2!}=\left(\frac{2}{\pi}\right)^{4N}.
\end{align}
By \eqref{eq:1.11}, it  now suffices to show that
\begin{align*}
\frac{2}{(4N+1)!}+\frac{2^{4N+2}(2^{4N+2}-1)2}{\left(2\pi\right)^{4N+2}\left(1-2^{1-2(2N+1)}\right)}<\left(\frac{2}{\pi}\right)^{4N},
\end{align*}
which can be rearranged as
\begin{align*}
\left(\frac{2}{\pi}\right)^{4N+2}\frac{1}{(4N+1)!}+\frac{1}{2^{4N+2}-2}<\frac{\pi^2}{8}-1,\qquad N\geq1.
\end{align*}
Noting that the sequence
\begin{align*}
y_N:=\left(\frac{2}{\pi}\right)^{4N+2}\frac{1}{(4N+1)!}+\frac{1}{2^{4N+2}-2},\qquad N\geq1
\end{align*}
is strictly decreasing, we see that
\begin{align*}
y_N
\leq y_1=\frac{\pi^4}{48384}<\frac{\pi^2}{8}-1
\end{align*}
holds true for $N\geq1$, since
\begin{align*}
\frac{\pi^4}{48384}=0.00201325\ldots,\quad \frac{\pi^2}{8}-1=0.23370055\ldots\ .
\end{align*}
 This proves the claim \eqref{Huygens-conjecture-rightnclaim}. Hence, \eqref{Huygens-conjecture-rightn=2N} holds, which shows that the second inequality in \eqref{ThmHuygensInequalityConjecture} holds for $n=2N$.
Thus, the second inequality in \eqref{ThmHuygensInequalityConjecture} holds for $0<x<\pi/2$ and $n\geq2$. The proof of Theorem \ref{Thm-Huygens-Inequality-Conjecture} is complete.
\end{proof}

\begin{theorem}\label{Wilker-Inequality-ChenConjecture}
For $0<x<\pi/2$, we have
\begin{align}\label{Wilker-Inequality-ChenCorrection}
2+\left(\frac{8}{45}-\frac{8}{945}x^2+ax^{4}\right)x^{3}\tan x&<\left(\frac{\sin x}{x}\right)^{2}+\frac{\tan x}{x}\nonumber\\
&<2+\left(\frac{8}{45}-\frac{8}{945}x^2+bx^{4}\right)x^{3}\tan x
\end{align}
with the best possible constants
\begin{align}\label{Wilker-Inequality-ChenCorrectionBestConstants}
a=\frac{16}{14175}=0.001128\ldots \quad\text{and}  \quad  b=\frac{241920-2688\pi^4+32\pi^6}{945\pi^8}=0.001209\ldots.
\end{align}
\end{theorem}

\begin{proof}
The inequality \eqref{Wilker-Inequality-ChenCorrection} can be written as
\begin{align*}
a<f(x)<b,
\end{align*}
where
\begin{align*}
f(x)=\frac{1}{x^4}\left(\frac{\left(\frac{\sin x}{x}\right)^{2}+\frac{\tan x}{x}-2}{x^{3}\tan x}-\left(\frac{8}{45}-\frac{8}{945}x^2\right)\right).
\end{align*}
Direct computations yield
\begin{align*}
\lim_{x\to0}f(x)=\frac{16}{14175} \quad\text{and}  \quad  \lim_{x\to\pi/2}f(x)=\frac{241920-2688\pi^4+32\pi^6}{945\pi^8}.
\end{align*}
In order to prove  \eqref{Wilker-Inequality-ChenCorrection}, it suffices to show that $f(x)$ is strictly increasing on $(0, \pi/2)$.

Differentiation yields
\begin{align*}
f'(x)=\frac{g(x)}{945x^{10}\sin^{2}x},
\end{align*}
with
\begin{align*}
g(x)&=6615x^2\sin(2x)-8505\sin^{3}x\cos x-8505x+1890x^3\\
&\quad+10395x\cos^{2}x-1890x\cos^{4}x+672x^5\sin^{2}x-16x^7\sin^{2}x\\
&=6615x^2\sin(2x)-8505\left(\frac{1}{4}\sin (2x)-\frac{1}{8}\sin(4x)\right)-8505x+1890x^3\\
&\quad+\frac{10395}{2}x\Big(1+\cos(2x)\Big)-1890x\left(\frac{1}{8}\cos(4x)+\frac{1}{2}\cos(2x)+\frac{3}{8}\right)\\
&\quad+336x^5\Big(1-\cos(2x)\Big)-8x^7\Big(1-\cos(2x)\Big)\\
&=\frac{16}{495}x^{13}+\frac{496}{61425}x^{15}-\frac{64}{26325}x^{17}+\sum_{n=9}^{\infty}(-1)^{n-1}u_{n}(x),
\end{align*}
where
\begin{align*}
u_n(x)&=\Big(945n\cdot2^{2n-1}-16065\cdot2^{2n-2}+16n^7-112n^6+952n^5\\
&\qquad\qquad-1960n^4+889n^3+13727n^2-2172n\Big)\frac{2^{2n}x^{2n+1}}{(2n+1)!},\qquad n\geq9.
\end{align*}

Direct computation yields
\begin{align*}
\frac{u_{n+1}(x)}{u_n(x)}=\frac{8x^2p_n}{q_n},
\end{align*}
where
\begin{align*}
p_n=(1890n-16065)4^n+16n^7+616n^5+1680n^4+889n^3+12810n^2+24309n+11340
\end{align*}
and
\begin{align*}
q_n&=(n+1)(2n+3)\Big((1890n-16065)4^n+64n^7-448n^6+3808n^5\\
&\qquad\qquad\qquad\qquad\qquad-7840n^4+3556n^3+54908n^2-8688n\Big).
\end{align*}
Noting that $8(\pi/2)^2<20$, we find that for $0<x<\pi/2$ and $n\geq9$,
\begin{align*}
\frac{u_{n+1}(x)}{u_n(x)}<\frac{8(\pi/2)^2p_n}{q_n}<\frac{20p_n}{q_n}<1,
\end{align*}
since\footnote{The inequality \eqref{Proof-since-Wilker-nequality} is proved in the appendix.}
\begin{align}\label{Proof-since-Wilker-nequality}
q_n-20p_n>0 \quad \text{for}\quad n\geq9.
\end{align}
Therefore, for fixed $x\in(0, \pi/2)$, the sequence $n\mapsto u_n(x)$ is strictly decreasing for $n\geq9$. Hence, for $0 <x <\pi/2$,
\begin{align*}
g(x)>\frac{16}{495}x^{13}+\frac{496}{61425}x^{15}-\frac{64}{26325}x^{17}=\frac{16}{495}x^{13}+\frac{16x^{15}(93-28x^2)}{184275}>0.
\end{align*}
We then obtain that $f'(x)>0$  for $0 <x <\pi/2$. The proof of Theorem \ref{Wilker-Inequality-ChenConjecture} is complete.
\end{proof}
 Following the same method used in the proof of Theorem \ref{Wilker-Inequality-ChenConjecture},
we can prove the following theorem.
\begin{theorem}\label{Wilker-Inequality-ChenThm2}
For $0<x<\pi/2$, we have
\begin{align}\label{Wilker-Chen-typeInequality}
2+\left(\frac{2}{45}-\frac{2}{315}x^2-\alpha x^{4}\right)x^{3}\tan x&<\left(\frac{x}{\sin x}\right)^{2}+\frac{x}{\tan x}\nonumber\\
&<2+\left(\frac{2}{45}-\frac{2}{315}x^2-\beta x^{4}\right)x^{3}\tan x,
\end{align}
\begin{align}\label{Huygens-Chen-typeInequality}
3+\left(\frac{3}{20}+\frac{1}{280}x^2+\lambda x^{4}\right)x^{3}\tan x&<2\left(\frac{\sin x}{x}\right)+\frac{\tan x}{x}\nonumber\\
&<3+\left(\frac{3}{20}+\frac{1}{280}x^2+\mu x^{4}\right)x^{3}\tan x
\end{align}
and
\begin{align}\label{HuygensType-Chen-typeInequality}
3+\left(\frac{1}{60}-\frac{1}{280}x^2-\rho  x^{4}\right)x^{3}\tan x&<2\left(\frac{x}{\sin x}\right)+\frac{x}{\tan x}\nonumber\\
&<3+\left(\frac{1}{60}-\frac{1}{280}x^2-\varrho  x^{4}\right)x^{3}\tan x,
\end{align}
with the best possible constants
\begin{align}\label{Wilker-Chen-typeInequalityBest}
\alpha=\frac{224-8\pi^2}{315\pi^4}=0.004727\ldots, \quad \beta=\frac{4}{1575}=0.002539\ldots,
\end{align}
\begin{align}\label{Huygens-Chen-typeInequalityBest}
\lambda=\frac{23}{33600}=0.000684\ldots, \quad  \mu=\frac{17920-168\pi^4-\pi^6}{70\pi^8}=0.000894\ldots
\end{align}
and
\begin{align}\label{Huygens-Chen-typeInequalityBestconstants}
\rho=\frac{56-3\pi^2}{210\pi^4}=0.0012901\ldots, \quad \varrho=\frac{83}{100800}=0.0008234\ldots.
\end{align}
\end{theorem}

\begin{proof}
We only prove inequality \eqref{HuygensType-Chen-typeInequality}. The proofs of \eqref{Wilker-Chen-typeInequality} and \eqref{Huygens-Chen-typeInequality} are  analogous.
The inequality \eqref{HuygensType-Chen-typeInequality} can be written as
\begin{align*}
\rho>F(x)>\varrho,
\end{align*}
where
\begin{align*}
F(x)=\frac{1}{x^4}\left(-\frac{2\left(\frac{x}{\sin x}\right)+\frac{x}{\tan x}-3}{x^{3}\tan x}+\left(\frac{1}{60}-\frac{1}{280}x^2\right)\right).
\end{align*}
Direct computations yield
\begin{align*}
\lim_{x\to0}F(x)=\frac{83}{100800}  \quad\text{and}  \quad   \lim_{x\to\pi/2}F(x)=\frac{56-3\pi^2}{210\pi^4}.
\end{align*}
In order to prove \eqref{HuygensType-Chen-typeInequality}, it suffices to show that $F(x)$ is strictly increasing on $(0, \pi/2)$.

Differentiation yields
\begin{align*}
F'(x)=\frac{G(x)}{420x^{8}\sin^{3}x},
\end{align*}
with
\begin{align*}
G(x)&=2520x\sin(2x)+(2520x-3x^5+28x^3)\sin x\cos^{2}x\\
&\quad+(3x^5-1260x-28x^3)\sin x+(840x^2-8820)\cos x+840x^2\cos^{2}x\\
&\quad+8820\cos^{3}x+840x^2\\
&=2520x\sin(2x)+(2520x-3x^5+28x^3)\left(\frac{1}{4}\sin x+\frac{1}{4}\sin (3x)\right)\\
&\quad+(3x^5-1260x-28x^3)\sin x+(840x^2-8820)\cos x+420x^2\Big(1+\cos(2x)\Big)\\
&\quad+8820\left(\frac{1}{4}\cos(3x)+\frac{3}{4}\cos x\right)+840x^2\\
&=\sum_{n=6}^{\infty}(-1)^{n}U_n(x),
\end{align*}
where
\begin{align*}
U_n(x)&=\Big((8n^5-40n^4+238n^3-302n^2-33924n+178605)9^n-(34020n^2+187110n)4^n\\
&\quad-5832n^5+29160n^4-64638n^3-215298n^2+222588n-178605\Big)\frac{x^{2n}}{81\cdot(2n)!}.
\end{align*}

Direct computation yields
\begin{align*}
\frac{U_{n+1}(x)}{U_n(x)}=\frac{9x^2P_n}{2Q_n},
\end{align*}
where
\begin{align*}
P_n&=(8n^5+158n^3+252n^2-33934n+144585)9^n-(15120n^2+113400n+98280)4^n\\
&\quad-648n^5-32508n^2-34938n-23625-702n^3
\end{align*}
and
\begin{align*}
Q_n&=(2n+1)(n+1)\Big((8n^5-40n^4+238n^3-302n^2-33924n+178605)9^n\\
&\qquad\qquad\qquad\qquad-(34020n^2+187110n)4^n-5832n^5+29160n^4-64638n^3\\
&\qquad\qquad\qquad\qquad-215298n^2+222588n-178605\Big).
\end{align*}
Noting that $\frac{9}{2}\left(\frac{\pi}{2}\right)^2<12$, we find that for $0<x<\pi/2$ and $n\geq6$,
\begin{align*}
\frac{U_{n+1}(x)}{U_n(x)}<\frac{9(\pi/2)^2P_n}{2Q_n}<\frac{12P_n}{Q_n}<1,
\end{align*}
since\footnote{The inequality \eqref{Proof-since-Huygens-nequality} is proved in the appendix.}
\begin{align}\label{Proof-since-Huygens-nequality}
Q_n-12P_n>0 \quad \text{for}\quad n\geq6.
\end{align}
Therefore, for fixed $x\in(0, \pi/2)$, the sequence $n\mapsto U_n(x)$ is strictly decreasing for $n\geq6$. Hence, we have
\begin{align*}
G(x)>0, \qquad 0<x<\frac{\pi}{2}.
\end{align*}
We then obtain that $F'(x)>0$  for $0 <x <\pi/2$. Hence, the inequality \eqref{HuygensType-Chen-typeInequality} holds with the best possible constants given in \eqref{Huygens-Chen-typeInequalityBestconstants}. The proof is complete.
\end{proof}
\begin{remark}
The upper bound in \eqref{Wilker-Chen-typeInequality} is sharper than the upper bound in \eqref{secondWilker-1}.
The inequalities \eqref{Huygens-Chen-typeInequality} are sharper than the inequalities \eqref{Huygens-generalization-ineq-1}.
\end{remark}
\begin{remark}
Chen and Paris \cite{Chen-Paris-submissionII} showed that for $0<x<\pi/2$,
\begin{align}\label{Thm-Huygens-type-inequality1}
3+\theta_1x^3\tan x<2\left(\frac{x}{\sin x}\right)+\frac{x}{\tan x}<3+\theta_2x^3\tan x
\end{align}
with the best possible constants
\begin{align*}
\theta_1=0\quad \text{and}\quad \theta_2=\frac{1}{60}.
\end{align*}
The double inequality \eqref{HuygensType-Chen-typeInequality}  is an improvement on the double inequality  \eqref{Thm-Huygens-type-inequality1}.
\end{remark}
\vskip 4mm

\begin{center}
{\bf Appendix A:  Proof of \eqref{Monotinic-ratio2}}
\end{center}
\setcounter{section}{1}
\setcounter{equation}{0}
\renewcommand{\theequation}{\Alph{section}.\arabic{equation}}

Noting that $\pi^2<10$, in order to prove   \eqref{Monotinic-ratio2}, it suffices to show that  for $k\geq2N+1$,
\begin{align}\label{Monotinic-ratio2proof}
&16^{k+N+1}-(8\cdot256^N+64^N)4^{k+1}+8\cdot1024^N-10(4^k-16^N)(4^{k+1}-16^N)\nonumber\\
&\quad=\Big((4^{2N+2}-40)\cdot4^k+50\cdot16^N-32\cdot256^N-4\cdot64^N\Big)4^k+(8\cdot1024^N-10\cdot256^N)\nonumber\\
&\quad>0.
\end{align}
We see that  for $k\geq2N+1$,
\begin{align*}
&(4^{2N+2}-40)\cdot4^k+50\cdot16^N-32\cdot256^N-4\cdot64^N\\
&\quad>(4^{2N+2}-40)\cdot4^{2N+1}+50\cdot16^N-32\cdot256^N-4\cdot64^N\\
&\quad=224\cdot256^N-590\cdot16^N-4\cdot64^N>0
\end{align*}
and
\begin{align*}
8\cdot1024^N-10\cdot256^N>0.
\end{align*}
Hence, \eqref{Monotinic-ratio2proof} holds for $k\geq2N+1$.

\vskip 4mm

\begin{center}
{\bf Appendix B: Proof of \eqref{Proof-since-Wilker-nequality}}
\end{center}
\setcounter{section}{2}
\setcounter{equation}{0}
\renewcommand{\theequation}{\Alph{section}.\arabic{equation}}
\begin{align*}
q_n-20p_n&=\Big(3780n^3-22680n^2-112455n+273105\Big)4^n+128n^9-576n^8+5248n^7\\
&\quad+2016n^6-32984n^5+70476n^4+250052n^3-134916n^2-512244n-226800\\
&=\Big(179550+397845(n-9)+79380(n-9)^2+3780(n-9)^3\Big)4^n\\
&\quad+49648561200+46968464520(n-9)+19975332000(n-9)^2\\
&\quad+5019956996(n-9)^3+822741108(n-9)^4+91303912(n-9)^5\\
&\quad+6864480(n-9)^6+337024(n-9)^7+9792(n-9)^8+128(n-9)^9\\
&>0 \quad \text{for}\quad n\geq9.
\end{align*}

\vskip 2mm

\begin{center}
{\bf Appendix C: Proof of \eqref{Proof-since-Huygens-nequality}}
\end{center}
\setcounter{section}{3}
\setcounter{equation}{0}
\renewcommand{\theequation}{\Alph{section}.\arabic{equation}}

We now show that for $ n\geq6$,
\begin{align*}
Q_n-12P_n&=(16n^7-56n^6+268n^5-70412n^3+252112n^2+909099n-1556415)9^n\\
&\quad+70n^4\cdot9^n-(68040n^4+476280n^3+413910n^2-1173690n-1179360)4^n\\
&\quad-(11664n^7-40824n^6+39852n^5+595350n^4+256932n^3-485352n^2\\
&\qquad-106029n-104895)>0.
\end{align*}
 It suffices to show that  for $ n\geq6$,
\begin{align}\label{Proof-since-Huygens-nequality1}
\left(\frac{9}{4}\right)^n>A_n,
\end{align}
where
\begin{align*}
A_n=\frac{68040n^4+476280n^3+413910n^2-1173690n-1179360}{16n^7-56n^6+268n^5-70412n^3+252112n^2+909099n-1556415},
\end{align*}
and
\begin{align}\label{Proof-since-Huygens-nequality2}
\cdot9^n&>\frac{1}{70n^4}\Big(11664n^7-40824n^6+39852n^5+595350n^4+256932n^3-485352n^2\nonumber\\
&\qquad\qquad-106029n-104895\Big).
\end{align}

By induction with respect to  $n$,  we can prove the inequalities \eqref{Proof-since-Huygens-nequality1} and \eqref{Proof-since-Huygens-nequality2}. Here, we only prove the inequality \eqref{Proof-since-Huygens-nequality1}. The proof of \eqref{Proof-since-Huygens-nequality2} is  analogous.

For $n=6$ in \eqref{Proof-since-Huygens-nequality1}, we find  that
\begin{align*}
\left(\frac{9}{4}\right)^6=\frac{531441}{4096} =129.746\ldots\quad \text{and}\quad A_6=\frac{3138660}{27229}=115.269\ldots.
\end{align*}
This shows that \eqref{Proof-since-Huygens-nequality1} holds  for $n=6$.
\par
Now we assume  that \eqref{Proof-since-Huygens-nequality1} holds
for some $n\geq6$. Then,  for $n\mapsto n+1$ in \eqref{Proof-since-Huygens-nequality1}, by using the induction hypothesis,
we have
\begin{align*}
&\left(\frac{9}{4}\right)^{n+1}-A_{n+1}>\frac{9}{4}A_n-A_{n+1}=\frac{2835R_n}{2S_{n}T_{n}},
\end{align*}
where
\begin{align*}
R_n&=960n^{11}+12384n^{10}+73088n^9+256200n^8-3508908n^7-22121984n^6+50474996n^5\\
&\quad+274552068n^4-445858781n^3-777353865n^2+997107984n-660306024\\
&=878926761468+1894841991720(n-6)+1695853296525(n-6)^2\\
&\quad+849645117283(n-6)^3+268187103036(n-6)^4+56595283460(n-6)^5\\
&\quad+8234103112(n-6)^6+835076820(n-6)^7+58479432(n-6)^8+2716928(n-6)^9\\
&\quad+75744(n-6)^{10}+960(n-6)^{11},
\end{align*}
\begin{align*}
S_n&=16n^7-56n^6+268n^5-70412n^3+252112n^2+909099n-1556415\\
&=1715427+679323(n-6)+1087672(n-6)^2+509908(n-6)^3+98760(n-6)^4\\
&\quad+10348(n-6)^5+616(n-6)^6+16(n-6)^7
\end{align*}
and
\begin{align*}
T_n&=16n^7+56n^6+268n^5+1060n^4-68292n^3+43052n^2+1203203n-465388\\
&=4102070+4834979(n-6)+3323012(n-6)^2+1021308(n-6)^3+160300(n-6)^4\\
&\quad+14380(n-6)^5+728(n-6)^6+16(n-6)^7.
\end{align*}
Hence, we have
\begin{align*}
\left(\frac{9}{4}\right)^{n+1}>A_{n+1}.
\end{align*}
Thus, by  the principle of  mathematical induction,  the inequality \eqref{Proof-since-Huygens-nequality1} holds  for $n\geq6$.

 \vskip 4mm

\section*{Acknowledgement} Some computations in this paper
were performed using Maple software.

\enddocument